\documentclass[11pt,a4paper,fleqn]{article}
\usepackage{a4wide,amsfonts,amsmath,latexsym,amssymb,euscript,graphicx,units,mathrsfs,color}

\usepackage[usenames,dvipsnames]{xcolor}
\usepackage{amsmath}
\usepackage{amssymb}
\usepackage{geometry}

\newtheorem{theorem}{Theorem}[section]

\newtheorem{definition}[theorem]{Definition}
\newtheorem{example}[theorem]{Example}

\newtheorem{lemma}[theorem]{Lemma}

\newtheorem{remark}[theorem]{Remark}

\newenvironment{proof}[1][Proof]{\noindent\textbf{#1.} }{\ \rule{0.5em}{0.5em}}
\begin{document}

\title{\bf Exchangeable Hoeffding decompositions over finite sets: a combinatorial characterization and counterexamples}
\author{Omar EL-DAKKAK, Giovanni PECCATI and Igor PR\"{U}NSTER \\
Universit\'{e} Paris-Ouest, Nanterre, La D\'{e}fense,\\
Universit\'{e} du Luxembourg and\\
Universit\`{a} degli Studi di Torino}
\date{}
\maketitle

\begin{abstract} We study Hoeffding decomposable exchangeable sequences with values in a finite set $D=\{d_{1},\ldots,d_{K}\}$. We provide a new combinatorial characterization of Hoeffding decomposability and use this result to show that, for every $K\geq 3$, there exists a class of neither P\'{o}lya nor i.i.d. $D$-valued exchangeable sequences that are Hoeffding decomposable. The construction of such sequences is based on some ideas appearing in Hill, Lane and Sudderth [1987] and answers a question left open in El-Dakkak and Peccati [2008].




\end{abstract}

\section{Introduction and framework}
\subsection{Overview}\label{ss:ovw}

Let $\mathbf{X}_{[1,\infty)}:=\{X_n:n\geq 1\}$ be a sequence of square-integrable random variables with values in some Polish space. We say that $\mathbf{X}_{[1,\infty)}$ is \textit{Hoeffding-decomposable} if every square-integrable symmetric statistic of any $n$-subvector of $\mathbf{X}_{[1,\infty)}$, for every $n\geq 2$, can be uniquely represented as an orthogonal sum of $n$ $U$-statistics with degenerate kernels of increasing order. The classic notion of `degeneracy' that is needed in this context is formally introduced in formula (\ref{comp-degen}) below.


Since their discovery in the landmark paper by Hoeffding [1948], Hoeffding decompositions in the case of i.i.d. sequences have been successfully applied in a variety of frameworks, e.g.: linear rank statistics (Hajek $[1968]$), jackknife estimators (Karlin and Rinott $[1982]$), covariance analysis of symmetric statistics (Vitale $[1992]$), convergence of $U$-processes (Arcones and Gin\'{e} $[1993]$), asymptotic problems in geometric probability (Avram and Bertsimas [1993]), Edgeworth expansions (Bentkus, G\"{o}tze and van Zwet $[1997]$), and tail estimates for $U$-statistics (Major $[2005]$). See also Koroljuk and Borovskich [1994] and references therein.


Outside the i.i.d. framework, Hoeffding decompositions have been notably applied to study sampling without replacement from finite populations. The first analysis in this direction can be found in Zhao and Chen $[1990]$. Bloznelis and G\"{o}tze $[2001, 2002]$ generalized these results in order to characterize the asymptotic normality of symmetric statistics based on sampling without replacement (when the size of the population diverges to infinity), as well as to obtain explicit Edgeworth expansions. In Bloznelis $[2005]$, Hoeffding-type decompositions are explicitly computed for statistics depending on extractions without replacement from several distinct populations.


In Peccati $[2003, 2004, 2008]$ the theory of Hoeffding decompositions was extended to the framework of general exchangeable (infinitely extendible) random sequences. In Peccati $[2004]$ it was shown that the class of Hoeffding decomposable exchangeable sequences coincides with the collection of weakly independent sequences, and that the class of weakly independent (and, therefore, Hoeffding decomposable) sequences contains the family of generalized P\'{o}lya urn sequences (see, e.g., Blackwell and MacQueen $[1973]$ or Pitman $[1996]$). The connection with P\'{o}lya urns was further exploited in Peccati $[2008]$, where Hoeffding-type decompositions were used in order to establish several new spectral properties of Ferguson-Dirichlet processes (Ferguson $[1973]$), such as for instance a chaotic representation property.


In El-Dakkak and Peccati $[2008],$ the results established in Peccati $[2004]$ were enriched and completed in two directions. On the one hand, it was proved that a (non deterministic) infinite exchangeable sequence with values in $\{0,1\}$ is Hoeffding decomposable if and only if it is either a P\'{o}lya sequence or i.i.d.. This result connects {\it de facto} the seemingly unrelated notions of a Hoeffding decomposable sequence and of a {\it urn process}, a concept thoroughly studied in Hill, Lane and Sudderth $[1987]$. For the sake of completeness, it is worth recalling that, following Hill, Lane and Sudderth [1987], an exchangeable sequence $\mathbf{X}_{[1,\infty)}$ will be termed {\it deterministic} if $\mathbb{P}[X_k = X_1,\, \forall k\geq 2] = 1$. On the other hand, and using different techniques, a partial characterization of Hoeffding decomposable exchangeable sequences with values in a finite set with more than two elements was obtained. While not being as exhaustive as the one in the two-color case, this characterization was used to prove that P\'{o}lya urns are the only Hoeffding decomposable sequences within a large class of exchangeable sequences. Such a family of exchangeable sequences is defined in terms of their directing (or de Finetti) measure, which is obtained by normalizing vectors of infinitely divisible (positive) independent random variables (see Regazzini, Lijoi and Pr\"unster [2003] and James, Lijoi and Pr\"unster [2006]). See Lijoi and Pr\"unster (2010) for an overview of their use in Bayesian Statistics.

\medskip

\noindent Therefore, the analysis carried out in El-Dakkak and Peccati [2008] left the following question unanswered:
\begin{align*}
\hbox{{\bf Problem A:}} \quad  & \hbox{{\sl Are P\'olya and i.i.d. sequences the only infinite non deterministic Hoeffding}}\\
 & \hbox{{\sl decomposable sequences with values in a finite set with $\geq 3$ elements ?}}
\end{align*}

\noindent We shall give a negative answer to Problem A. This is surprising given the above mentioned positive characterization might somehow lead to conjecture the opposite and, hence, makes the present result even more remarkable. In fact, the negative answer is obtained by explicitly building a class of {\it neither P\'olya nor i.i.d.} yet Hoeffding decomposable exchangeable sequences with values in a finite set with strictly more than two elements. See Theorem \ref{t:main} below for a precise statement. Interestingly, this class turns out to be a generalization of a counterexample appearing in Hill, Lane and Sudderth [1987, p. 1591], showing that, unlike in the two-color case (see Theorem \ref{t:mainhls}),  there exist non-deterministic exchangeable $3$-color urn processes that are neither P\'olya nor i.i.d. sequences. Our fundamental tool is a new combinatorial characterization, stated in Theorem  \ref{omaraus_K}, of the system of predictive probabilities associated with Hoeffding-decomposable exchangeable sequences taking values in an arbitrary finite set. This characterization, which is of independent interest, represents a generalization of a crucial combinatorial statement proved in Proposition 4 of El-Dakkak and Peccati $[2008]$.


In what follows, we recall some well-known facts about the main notions used in the sequel, namely: exchangeability (Section \ref{ss:ex}), Hoeffding decomposability  (Section \ref{ss:hd}), weak independence  (Section \ref{ss:wi}) and urn processes  (Section \ref{ss:up}).


For further details about exchangeability, urn processes and Hoeffding decompositons, the reader is referred e.g. to Aldous $[1983]$, Hill, Lane and Sudderth $[1987]$, Peccati $[2004]$  and  El-Dakkak and Peccati $[2008]$.

\begin{remark}{\rm
Every exchangeable sequence $\{X_1, X_2,...\}$ considered in this paper is assumed to take values in some finite set $D$. In particular, every random variable of the type $F = \varphi(X_1,...,X_n)$, $n\geq 1$, is automatically bounded.
}
\end{remark}

\subsection{Exchangeability}\label{ss:ex}

For every $n\geq 2$, we denote by $\mathfrak{S}_{n}$
the group of all permutations of the set $\left[ n\right] =\left\{
1,...,n\right\} .$ A vector $\left( X_{1},...,X_{n}\right) $ of $D$-valued
random variables is said to be \textit{exchangeable} if, for all $\mathbf{x}%
_{n}=\left( x_{1},...,x_{n}\right) \in D^{n}$ and all $\pi \in \mathfrak{S}%
_{n},$%
\begin{equation*}
\mathbb{P}\left( X_{1}=x_{1},...,X_{n}=x_{n}\right) =\mathbb{P}\left(
X_{1}=x_{\pi \left( 1\right) },...,X_{n}=x_{\pi \left( n\right) }\right) .
\end{equation*}%
A $D$-valued infinite sequence $\mathbf{X}_{\left[ 1,\infty \right) }$ is exchangeable
if every $n$-subvector of $\mathbf{X}_{\left[ 1,\infty \right) }$ is
exchangeable. Let $\Pi _{D}$ denote the set of all probability measures
on $D$. By the de Finetti representation theorem (see Aldous $[1983]$), an infinite sequence $%
\mathbf{X}_{\left[ 1,\infty \right) }$ with values in $D=\left\{
d_{1},...,d_{K}\right\} $ is exchangeable
if and only if there exists a unique probability measure $\gamma $ on $\Pi _{D}$
(called {\it directing} or the {\it de Finetti measure} associated with the sequence $\mathbf{X}_{\left[
1,\infty \right) }$) such that, for all $n\geq 2$ and all $\left( x_{1},...,x_{n}\right)
\in D^{n},$%
\begin{equation}
\mathbb{P}\left( X_{1}=x_{1},...,X_{n}=x_{n}\right) =\int_{\Pi
_{D}}\prod\limits_{j=1}^{n}p\left\{ x_{j}\right\} \gamma \left( \text{d}%
p\right) ,  \label{definetti2}
\end{equation}%
where the elements of $\Pi _{D}$ are written in the form $p:= \left\{ p\left\{
d_{i}\right\} :i=1,...,K\right\} .$ In other words, the de Finetti
representation theorem states that a sequence of random variables is
exchangeable if and only if it is a mixture of i.i.d. random sequences
with values in $D.$

\medskip

Plainly, any probability measure $p\in \Pi_D$ can be parametrized
in terms of the simplex $$\Theta _{K-1}:= \left\{ \left( \theta
_{1},...,\theta _{K-1}\right) :\theta _{h}\geq 0,\text{ }h=1,...,K-1\text{ and }%
\sum_{h=1}^{K-1}\theta _{h}\leq 1\right\} ,$$ by setting $p\left\{
d_{1}\right\} =\theta _{1},...,p\left\{ d_{K-1}\right\} =\theta _{K-1}$ and $%
p\left\{ d_{K}\right\} =1-\sum_{h=1}^{K-1}\theta _{h}.$ With this convention, the representation in (\ref{definetti2}) can be rewritten as

\begin{equation}
\mathbb{P}\left( X_{1}=x_{1},...,X_{n}=x_{n}\right) =\int_{\Theta _{K-1}}\left( \Pi
_{j=1}^{K-1}\theta _{j}^{i_{j}}\right) \left( 1-\Sigma _{h=1}^{K-1}\theta
_{h}\right) ^{i_{K}}\gamma \left( \text{d}\theta _{1},...,\text{d}\theta
_{K-1}\right) ,  \label{definettibis2}
\end{equation}%
where (with an abuse of notation) we have identified $\gamma $ with its image through the canonical bijection between $\Pi_D$ and $\Theta _{K-1}$, and where $i_{j}:=
\sum_{v=1}^{n}\mathbf{1}\left( x_{v}=d_{j}\right) ,$ $j=1,...,K.$ Note that, when $K=2$,  (\ref{definettibis2}) becomes%
\begin{equation}\label{amsterdam}
\mathbb{P}\left( X_{1}=x_{1},...,X_{n}=x_{n}\right) =\int_{\left[ 0,1\right] }\theta
^{i}\left( 1-\theta \right) ^{n-i}\gamma \left( \text{d}\theta \right) ,
\end{equation}%
where $i=\sum_{v=1}^{n}\mathbf{1}\left( x_{v}=d_{1}\right) .$

\medskip

If there exists a vector $\alpha =\left( \alpha _{1},...,\alpha _{K}\right)
\in \left( 0,+\infty \right) ^{K}$ of strictly positive numbers such that
\begin{equation}
\gamma(\text{d}\theta_{1},\ldots,\text{d}\theta_{K-1}) =\frac{1}{B\left( \alpha \right) }\left( \Pi _{j=1}^{K-1}\theta
_{j}^{\alpha _{j}-1}\right) \left( 1-\Sigma _{h=1}^{K-1}\theta _{h}\right)
^{\alpha _{K}-1}\text{d}\theta _{1}\cdots\text{ d}\theta _{K-1},
\label{for:dirichletbis}
\end{equation}%
where $B\left( \alpha \right) := \Pi _{j=1}^{K}\Gamma \left( \alpha
_{j}\right) /\Gamma \left( \Sigma _{j=1}^{K}\alpha _{j}\right),$ and $%
\Gamma \left( \cdot \right) $ stands for the usual Gamma function, we say that $%
\gamma $ is a \textit{Dirichlet probability measure} and that $\mathbf{X}_{%
\left[ 1,\infty \right) }$ is a \textit{$K$-color P\'{o}lya sequence} with parameter $\alpha$. Specializing (\ref{for:dirichletbis}) to the case $K=2,$ one sees
immediately that the measure $\gamma $ in (\ref{amsterdam}) becomes a Beta distribution with parameters $\alpha_1,\, \alpha_2$. It follows that
an exchangeable sequence $\mathbf{X}_{\left[ 1,\infty \right) }$ is a
{two-color P\'{o}lya sequence} if and only if its de Finetti measure is a Beta
distribution.

\subsection{Hoeffding decomposability}\label{ss:hd}

Let us first introduce some notation.
For all $n\geq 1$ and all $1\leq u\leq n,$ we write $\left[ n\right]
=\left\{ 1,...,n\right\} $ and $\left[ u,n\right] =\left\{
u,u+1,...,n\right\} ,$ and set $\mathbf{X}_{\left[ n\right] }:=
\left( X_{1},X_{2},...,X_{n}\right) $ and $\mathbf{X}_{\left[ u,n\right]
}:= \left( X_{u},X_{u+1},...,X_{n}\right) .$ As in El-Dakkak and
Peccati $\left[ 2008\right] ,$ define, for all $n\geq 2,$ the sequence of
spaces%
\begin{equation*}
\left\{ SU_{k}\left( \mathbf{X}_{\left[ n\right] }\right) :k=0,...,n\right\}
,
\end{equation*}%
generated by symmetric $U$-statistics of increasing order, as follows: $%
SU_{0}\left( \mathbf{X}_{\left[ n\right] }\right) := \mathfrak{R}$
and, for all $k=1,...,n,$ $SU_{k}\left( \mathbf{X}_{\left[ n\right] }\right)
$ is the collection of all random variables of the type%
\begin{equation}
F\left( \mathbf{X}_{\left[ n\right] }\right) =\sum_{1\leq j_{1}<\cdot \cdot
\cdot <j_{k}\leq n}\varphi \left( X_{j_{1}},...,X_{j_{k}}\right) ,
\label{u-stat}
\end{equation}%
where $\varphi $ is a real-valued symmetric function from $D^{k}$ into $%
\mathfrak{R}.$ Any random variable $F$ as in (\ref{u-stat}) is called a $U$%
\textit{-statistic with symmetric kernel of order }$k.$ Since the collection $%
\left\{ SU_{k}\left( \mathbf{X}_{\left[ n\right] }\right)\right\}
$ is an increasing sequence of vector spaces such that $SU_{n}\left( \mathbf{X}_{%
\left[ n\right] }\right) =L_{s}\left( \mathbf{X}_{\left[ n\right] }\right)$
(where $L_{s}\left( \mathbf{X}_{\left[ n\right] }\right)$ is defined as the Hilbert space of all symmetric statistics $T\left( \mathbf{X}_{\left[ n%
\right] }\right) $ with respect to the inner product $\left\langle
T_{1},T_{2}\right\rangle := \mathbb{E}\left( T_{1}\left( \mathbf{X}_{%
\left[ n\right] }\right) T_{2}\left( \mathbf{X}_{\left[ n\right] }\right)
\right) ),$ one can meaningfully define the sequence of \textit{symmetric
Hoeffding spaces} associated with $\mathbf{X}_{\left[ n\right] }$, denoted by
\begin{equation*}
\left\{ SH_{k}\left( \mathbf{X}_{\left[ n\right] }\right) :k=0,...,n\right\},
\end{equation*}%
as follows: $SH_{0}\left(
\mathbf{X}_{\left[ n\right] }\right) := SU_{0}\left( \mathbf{X}_{%
\left[ n\right] }\right) =\mathfrak{R},$ and%
\begin{equation*}
SH_{k}\left( \mathbf{X}_{\left[ n\right] }\right) := SU_{k}\left(
\mathbf{X}_{\left[ n\right] }\right) \cap SU_{k-1}\left( \mathbf{X}_{\left[ n%
\right] }\right) ^{\bot },\qquad k=1,...,n,
\end{equation*}%
where all orthogonals (here and in the sequel) are taken in $L_{s}\left(
\mathbf{X}_{\left[ n\right] }\right) .$ The following representation is
therefore at hand%
\begin{equation*}
L_{s}\left( \mathbf{X}_{\left[ n\right] }\right)
=\bigoplus\limits_{k=0}^{n}SH_{k}\left( \mathbf{X}_{\left[ n\right] }\right)
,
\end{equation*}%
where `$\oplus $" stands for an orthogonal sum. This fact implies that,
for all $n\geq 2,$ every symmetric statistic $T\left( \mathbf{X}_{\left[ n%
\right] }\right) $ admits an almost-surely unique representation of the type:

\begin{equation}\label{e:hd1}
T\left( \mathbf{X}_{\left[ n\right] }\right) =\mathbb{E}\left( T\right)
+\sum_{k=1}^{n}F_{k}\left( \mathbf{X}_{\left[ n\right] }\right) ,
\end{equation}%
where the $F_{k}$'s are uncorrelated $U$-statistics such that, for all $%
k=1,...,n,$ there exists a symmetric kernel $\varphi _{k}$ of order $k$
satisfying%
\begin{equation}\label{e:hd2}
F_{k}\left( \mathbf{X}_{\left[ n\right] }\right) =\sum_{1\leq i_{1}<\cdot
\cdot \cdot <i_{k}\leq n}\varphi _{k}\left( X_{i_{1}},...,X_{i_{k}}\right) .
\end{equation}%

We are now in a position to recall the definition of Hoeffding decomposability for exchangebale sequences given in Peccati [2004].

\begin{definition}\label{d:hd} {\rm The exchangeable sequence $\mathbf{X}_{\left[ 1,\infty \right) }$ is said to be \textit{%
Hoeffding decomposable} if, for all $n\geq 2$ and all $k=1,...,n,$ the
following double implication holds: $F_{k}\in SH_{k}\left( \mathbf{X}_{\left[
n\right] }\right) $ if and only if the kernel $\varphi _{k}$ appearing in
its representation (\ref{e:hd1})--(\ref{e:hd2}) is \textit{completely degenerate}, that is
\begin{equation}
\mathbb{E}\left( \varphi \left( \mathbf{X}_{\left[ k\right] }\right) \text{ }%
|\text{ }\mathbf{X}_{\left[ 2,k\right] }\right) =0,\qquad a.s.\text{-}%
\mathbb{P}.  \label{comp-degen}
\end{equation}
}
\end{definition}

\subsection{Weak independence}\label{ss:wi}

Fix $n\geq 2$ and let $\mathcal{S}\left(
D^{n}\right) $ be the class of all symmetric real-valued functions on $%
D^{n}. $ Fix $\varphi \in \mathcal{S}\left( D^{n}\right) $ and define the
functions $\left[ \varphi \right] _{n,n-1}^{\left( n-1\right)
}:D^{n-1}\rightarrow \mathfrak{R}$ and $\left[ \varphi \right]
_{n,n-1}^{\left( n-u\right) }:D^{n-1}\rightarrow \mathfrak{R},$ $u=2,...,n$
as the unique mappings such that%
\begin{equation}
\left[ \varphi \right] _{n,n-1}^{\left( n-1\right) }\left( \mathbf{X}_{\left[
2,n\right] }\right) =E\left( \varphi \left( \mathbf{X}_{\left[ n\right]
}\right) \text{ }|\text{ }\mathbf{X}_{\left[ 2,n\right] }\right) ,\qquad a.s.%
\text{-}\mathbb{P},  \label{dentro}
\end{equation}%
and%
\begin{equation}
\left[ \varphi \right] _{n,n-1}^{\left( n-u\right) }\left( \mathbf{X}_{\left[
u+1,u+n-1\right] }\right) =E\left( \varphi \left( \mathbf{X}_{\left[ n\right]
}\right) \text{ }|\text{ }\mathbf{X}_{\left[ u+1,u+n-1\right] }\right)
,\qquad a.s.\text{-}\mathbb{P},  \label{fuori}
\end{equation}%
respectively. Exchangeability and symmetry imply that $D^{n-1}\to
\mathfrak{R}:\mathbf{x}\mapsto \left[ \varphi \right] _{n,n-1}^{\left(
n-1\right) }\left( \mathbf{x}\right) $ and $D^{n-1}\to \mathfrak{R}:%
\mathbf{x}\mapsto \left[ \varphi \right] _{n,n-1}^{\left( 0\right) }\left(
\mathbf{x}\right)$ (the latter being the function appearing (\ref{fuori})
written for $u=n$) are symmetric functions whereas, for $u=2,...,n-1,$ the
function $D^{n-1}\ni \left( x_{1},...,x_{n-1}\right) \mapsto \left[ \varphi %
\right] _{n,n-1}^{\left( n-u\right) }\left( x_{1},...,x_{n-1}\right) $ is
separately symmetric in the variables $\left( x_{1},...,x_{n-u}\right) $ and
$\left( x_{n-u+1},...,x_{n-1}\right) ,$ and not necessarily as a function on
$D^{n-1}.$ Recall that, given a function $f:D^{n}\rightarrow \mathfrak{R}$,
the canonical symmetrization of $f,$ denoted by $\tilde{f},$ is given by
\begin{equation*}
\tilde{f}\left( \mathbf{x}_{n}\right) =\frac{1}{n!}\sum_{\pi \in \mathfrak{S}%
_{n}}f\left( \mathbf{x}_{\pi \left( n\right) }\right) ,\,\,  \mathbf{x}_{n}\in D^{n}.%
\end{equation*}
Now define the sequence of vector
spaces%
\begin{equation*}
\Xi _{n}\left( \mathbf{X}_{\left[ 1,\infty \right) }\right) :=
\left\{ \varphi \in \mathcal{S}\left( D^{n}\right) :\left[ \varphi \right]
_{n,n-1}^{\left( n-1\right) }\left( \mathbf{X}_{\left[ 2,n\right] }\right)
=0\right\} ,\qquad n\geq 2,
\end{equation*}%
and the array of spaces%
\begin{equation*}
\tilde{\Xi}_{n,n-u}\left( \mathbf{X}_{\left[ 1,\infty \right) }\right)
:= \left\{ \varphi \in \mathcal{S}\left( D^{n}\right) :\widetilde{%
\left[ \varphi \right] }_{n,n-1}^{\left( n-u\right) }\left( \mathbf{X}_{%
\left[ u+1,u+n-1\right] }\right) =0\right\}
,\quad u=2,...,n,\text{ }n\geq 2,
\end{equation*}%
where $\widetilde{\left[ \varphi \right] }_{n,n-1}^{\left( n-u\right) }$ is
the canonical symmetrization of $\left[ \varphi \right] _{n,n-1}^{\left(
n-u\right) }.$

\begin{definition}{\rm We say that the sequence $\mathbf{X}_{\left[ 1,\infty \right)
}$ is \textit{weakly independent }if, for every $n\geq 2,$%
\begin{equation*}
\Xi _{n}\left( \mathbf{X}_{\left[ 1,\infty \right) }\right) \subset
\bigcap\limits_{u=2}^{n}\tilde{\Xi}_{n,n-u}\left( \mathbf{X}_{\left[
1,\infty \right) }\right) .
\end{equation*}%
In other words, weak independence occurs if, for every $n\geq 2$ and every $%
\varphi \in \mathcal{S}\left( D^{n}\right) ,$ the relation $\left[ \varphi \right]
_{n,n-1}^{\left( n-1\right) }\left( \mathbf{X}_{\left[ 2,n\right] }\right)
=0$ necessarily implies that $\widetilde{\left[ \varphi \right] }_{n,n-1}^{\left(
n-u\right) }\left( \mathbf{X}_{\left[ u+1,u+n-1\right] }\right) =0$ for all $%
u=2,...,n.$ For instance, when $n=3$ weak independence yields the following implication for every
symmetric $\varphi $ on $D^{2}$:%
\begin{equation*}
\mathbb{E}\left( \varphi \left( X_{1},X_{2}\right) \text{ }|\text{ }%
X_{2}\right) =0\quad \Rightarrow \quad \mathbb{E}\left( \varphi \left(
X_{1},X_{2}\right) \text{ }|\text{ }X_{3}\right) =0.
\end{equation*}
}
\end{definition}

The following statement contains some of the main findings in Peccati [2004] (Part I) and El-Dakkak and Peccati [2008] (Part II).
\begin{theorem}[Peccati, 2004; El-Dakkak and Peccati, 2008]\label{t:eldkp}
Let $\mathbf{X}_{\left[ 1,\infty \right)}$ be an exchangeable sequence of random variables with values in the finite set $D$.

\begin{itemize}

\item[\bf (I)] Assume that
\begin{equation}
SH_{k}\left( \mathbf{X}_{\left[ n\right] }\right) \neq \left\{ 0\right\}
,\qquad \forall k=1,...,n,\text{ }\forall n\geq 2.  \label{theorem0}
\end{equation}%
Then $\mathbf{X}_{\left[ 1,\infty \right)}$ is Hoeffding decomposable if and only if it is weakly independent.
\item[\bf (II)] Assume that $D = [2]$ and that $\mathbf{X}_{\left[ 1,\infty \right)}$ is non deterministic (so that (\ref{theorem0}) is automatically satisfied). Then, $\mathbf{X}_{\left[ 1,\infty \right)}$ is Hoeffding decomposable if and only if $\mathbf{X}_{\left[ 1,\infty \right)}$ is either a P\' olya sequence or an i.i.d. sequence.

\end{itemize}

\end{theorem}


\subsection{Urn processes and a result by Hill, Lane and Sudderth}\label{ss:up}
Let $\mathbf{X}_{[1,\infty)}:=\{X_n:n\geq 1\}$ be a sequence of $\{0,1\}$-valued random variables. Roughly speaking, $\mathbf{X}_{[1,\infty)}$ is a \textit{two-color urn process} if its probabilistic structure can be represented by successive drawings from an urn with changing composition. More precisely, consider an urn containing $r$ red balls and $b$ black balls, $r,b\in\{1,2,...\},$ and let $Y_0:=r/(r+b)$ denote the initial proportion of red balls in the urn. Suppose that a red ball is added with probability $f(Y_0)$ and that a black ball is added with probability $1-f(Y_0)$, where $f$ denotes a function from the unit interval into itself, and let $Y_1$ be the new proportion of red balls in the urn. Now, iterate the procedure to generate a sequence $(Y_0,Y_1,Y_2,\ldots).$ For all $n\geq 1,$ let $X_n$ denote the indicator of the event that the $n$-th ball added is red. The process $\mathbf{X}_{[1,\infty)}:=\{X_n:n\geq 1\}$ constructed in this manner is called a \textit{two-color urn process with initial composition $(r,b)$ and urn function $f$}. It is immediately seen that, for all $n\geq 1,$
$$\mathbb{P}\left(X_{n+1}=1\text{ }|\text{ }X_1,\ldots,X_n\right)=f(Y_n).$$
In other words, two-color urn processes are characterized by the fact that the conditional probability that, at stage ${n+1},$ a red ball is added depends uniquely on the proportion of red balls at stage $n$, via the function $f$.

A two-color urn process is said to be exchangeable if the sequence $\mathbf{X}_{[1,\infty)}$ is exchangeable (see Section 1.1). In particular, if $\mathbf{X}_{[1,\infty)}$ is a two-color urn process with initial composition $(r,b),$ and the identity map as urn function then (a) it is exchangeable and (b) the de Finetti measure of $\mathbf{X}_{[1,\infty)}$ is a Beta distribution with parameters $r$ and $b$. In that case, we will say that $\mathbf{X}_{[1,\infty)}$ is a \textit{two-color P\'{o}lya urn process}. Similarly, a two-color urn process, $\mathbf{X}_{[1,\infty)}$, with constant urn function, identically equal to $Y_0$ is (a) exchangeable and (b) has de Finetti measure equal to a point mass at $Y_0.$ In other words, $\mathbf{X}_{[1,\infty)}$ is a sequence of i.i.d. Bernoulli trials with parameter $Y_0$. Finally, a two-color urn process, $\mathbf{X}_{[1,\infty)},$ with urn function
$$f(x)=p\mathbf{1}_{\{X_0\}}(x)+\mathbf{1}_{(X_0,1]}(x), \qquad p\in(0,1),$$
is (a) exchangeable and (b) has de Finetti measure $\gamma=p\delta_{\{1\}}+(1-p)\delta_{\{0\}}.$ In that case, we will say that $\mathbf{X}_{[1,\infty)}$ is a \textit{deterministic urn process}. Note that such processes are characterized by the fact that the support of their de Finetti measure is $\{0\}\cup\{1\}.$

The following statement is the main result of Hill, Lane and Sudderth [1987]: it shows that the three classes described above are the only two-color exchangeable urn processes.

\begin{theorem}[Hill, Lane and Sudderth, 1987]\label{t:mainhls}
 Let $\mathbf{X}_{[1,\infty)}$ be an exchangeable non deterministic urn process with values in $\{0,1\}$. Then, $\mathbf{X}_{[1,\infty)}$ is either a two-color P\' olya urn process, or an i.i.d. Bernoulli sequence.
\end{theorem}

\medskip

\begin{remark}{\rm
In the parlance of the present article, a distinction is made between \textit{P\'{o}lya sequences} and \textit{P\'{o}lya urn processes}, the latter being a proper subset of the former: in fact, according to our definitions, a P\'{o}lya urn process is a P\'{o}lya sequence with de Finetti measure given by a Beta distribution whose parameters are integer-valued.
}
\end{remark}

\medskip

We now turn to the definition of multicolor urn processes. Consider an urn containing balls of $K$ colors, $K\in\{3,4,\ldots\}$, and suppose that it contains exactly $r_i$ balls of color $d_{i}$, respectively, $r_i\in\{1,2,\ldots\},$ $i=1,\ldots,K.$ Let $Y_0:=(Y_{0,1},\ldots,Y_{0,K})$ be the vector of initial proportions of balls of each color in the urn, where $Y_{0,i}:=\frac{r_i}{\sum_{k=1}^Kr_k}$ denotes the proportion of balls of color $d_{i}$, $i=1,\ldots,K$. Suppose that, at stage $1,$ a ball is added and that it is of color $d_{j}$ with probability $f_j(Y_{0,1},\ldots,Y_{0,K}),$ $j=1,\ldots,K,$ where the $f_j$'s are $[0,1]$-valued functions defined on the simplex
\begin{equation} \label{simplexurn}
S_K:=\left\{\mathbf{y}=(y_1,\ldots,y_K):\sum_{k=1}^Ky_k=1,\text{ }y_k\geq 0,\text{ }k=1,\ldots,K\right\},
\end{equation}
such that, for all $\mathbf{y}\in S_K,$ $\sum_{j=1}^Kf_j(\mathbf{y})=1.$ Let $Y_1:=(Y_{1,1},\ldots,Y_{1,K})$ be the new composition of the urn and iterate the process to generate a sequence $(Y_0,Y_1,Y_2,\ldots).$ For all $n\geq 1$, let $X_n$ be the $\{d_{1},\ldots,d_{K}\}$-valued random variable such that $X_n=d_{j}$ if and only if the ball added, at stage $n$, is of color $d_{j}$. The process $\mathbf{X}_{[1,\infty)}:=\{X_n:n\geq 1\}$ obtained in this manner is called a \textit{$K$-color urn process with initial composition $(r_1,\ldots,r_K)$ and urn function $f=(f_1,\ldots,f_K)$}, and we have, for all $n\geq 1$ and all $j\in\{1,\ldots,K\},$
$$\mathbb{P}\left(X_{n+1}=d_{j}\text{ }|\text{ }X_1,\ldots,X_n\right)=f_j(Y_n).$$
A $K$-color urn process with initial composition $(r_1,\ldots,r_K)$ and an urn function given by the identity map is (a) exchangeable and (b) has de Finetti measure given by a Dirichlet distribution of parameters $r_1,\ldots,r_K$. Such an urn process will be called a \textit{$K$-color P\'{o}lya urn process with initial composition $(r_1,\ldots,r_K)$}. Once more, the class of $K$-color P\'{o}lya urn processes is a proper subset of the class of $K$-color P\'{o}lya sequences. The following example, taken from p. 1591 of Hill, Lane and Sudderth [1987], shows that a neat result such as Theorem \ref{t:mainhls} cannot hold for exchangeable urn processes with values in sets with strictly more than two elements.

\begin{example}\label{ex:HLS}{\rm An urn contains three balls, 1 red, 1 black and 1 green. At each stage, a ball is drawn. If the ball is red, it is replaced and another red ball is added. If the ball is black or green, it is replaced, and a green or black ball is added, depending whether a fair coin falls head or tails. Attaching the labels 1, 2, 3, respectively, to the colors red, black and green, one sees immediately that the sequence $\{X_n : n\geq 1\}$, defined as $X_n=j$ ($j=1,2,3$) according to whether the $n$th ball added to the urn is of color $j$, is an exchangeable urn process with urn function given by $f_1({\bf y}) = y_1$ and $f_2({\bf y}) = f_3({\bf y}) = (y_2 + y_3)/2$. In particular, $\{X_n\}$ is not a P\'olya urn process. }\label{ex:hls}
\end{example}

The main achievement of the present paper is the proof that a generalization of the previous example provides examples of Hoeffding decomposable exchangeable sequences that are neither P\'olya nor i.i.d..

\subsection{Plan}\label{ss:plan}

Section \ref{s:mainresults} contains a discussion and the statement of our main result: Theorem \ref{t:main}. Section \ref{s:maintools} contains the main combinatorial tools and the novel combinatorial characterization that are needed throughout the present paper, whereas the proof of Theorem \ref{t:main} is provided in Section \ref{s:proof}.

\section{A remarkable class of exchangeable sequences}\label{s:mainresults}


To achieve the announced negative result we introduce a remarkable class of exchangeable sequences. As will be clear from its definition, this class generalizes the exchangeable sequence introduced by Hill, Lane and Sudderth [1987] that we recalled in Example~\ref{ex:HLS}. Let $K\geq 3$, and let $\mathbf{X}_{[1,\infty)}$ be an exchangeable sequence with values in $D=\{d_{1},\ldots,d_{K}\}$, whose de Finetti measure $\gamma$ is such that

\begin{equation}\label{genHLS}
\gamma(\text{d}\theta_{1},\ldots,\text{d}\theta_{K-1})=\frac{1}{B(\pi,\nu)}\delta_{\alpha_{1}(1-\theta_{1})}(\text{d}\theta_{2})\cdots\delta_{\alpha_{K-2}(1-\theta_{1})}(\text{d}\theta_{K-1})\theta_{1}^{\pi-1}(1-\theta_{1})^{\nu-1}\text{d}\theta_{1},
\end{equation}
where $\pi,\nu>0$ and $\alpha_{1},\ldots,\alpha_{K-2}>0$ are such that $\sum_{i=1}^{K-2}\alpha_{i}<1$.

\begin{remark}{\rm
Equation (\ref{genHLS}) defines the de Finetti measure of an exchangeable sequence that is neither i.i.d. nor P\' olya. In the sequel, we will refer to any such sequence as a $$HLS_{K}(\pi,\nu,\alpha_{1},\ldots,\alpha_{K-2}) \mbox{\ \ \it exchangeable \, sequence},$$ (or, simply, a $HLS_{K}$-exchangeable sequence, if the parameters $\pi,\nu,\alpha_{1},\ldots,\alpha_{K-2}$ need not be specified in a given context), with reference to the paper by Hill, Lane and Sudderth [1987]. In particular, as deduced from the discussion below, the case $HLS_{3}(1,2,\frac{1}{2})$ corresponds to the $3$-color urn sequence described in Example \ref{ex:hls}.

}
\end{remark}

When $\pi$ and $\nu$ are integer-valued, for any fixed $K$ all $HLS_{K}(\pi,\nu,\alpha_{1},\ldots,\alpha_{K-2})$ exchangeable sequences are non P\' olya exchangeable urn processes. To see this, it suffices to notice that any such sequence can be generated by means of an urn with initial composition $(\pi,\nu_{1},\ldots,\nu_{K-1})$, where the integers $\nu_i$ are such that $\nu=\sum_{i=1}^{K-1}\nu_{i}$, and with urn function $f=(f_{1},\ldots,f_{K})$ given by $f_{1}(\mathbf{y})=y_{1}$, $f_{j}(\mathbf{y})=\alpha_{j-1} \sum_{i=1}^{K-1} y_i$ ($j=2,...,K-1$), and $f_{K}(\mathbf{y})= (1-\alpha) \sum_{i=1}^{K-1} y_i$, with $\alpha=\sum_{i=1}^{K-2}\alpha_{i}$. An $HLS_{K}(\pi,\nu,\alpha_{1},\ldots,\alpha_{K-2})$ exchangeable urn process (i.e. with integer-valued $\pi$ and $\nu$) has consequently the following interpretation: suppose an urn contains initially $\pi$ balls of color $d_{1}$ and $\nu_{i-1}$ balls of color $d_{i}$, $i=2,\ldots,K$, with $\sum_{i=1}^{K-1}\nu_{i}=\nu$. The following random experiment is run at each stage: a ball is drawn, if it is of color $d_{1}$, it is replaced along with another of the same color. If the ball drawn is of color $d_{i}$, $i=2,\ldots,K$, it is replaced along with a ball of color $d_{j}$, with probability $t_{j}$, where $t_{j}=\alpha_{j-1}$, if $j=2,\ldots,K-1$ and $t_{j}=1-\alpha=1-\sum_{s=1}^{K-2}\alpha_{s}$, if $j=K$.

\begin{remark}{\rm In the above described explicit realization of a sequence of the type $HLS_{K}(\pi,\nu,\alpha_{1},\ldots,\alpha_{K-2})$, the initial decomposition of the index $\nu$ into integers $\nu_i$, $i=2,...,K$, is immaterial.
}
\end{remark}
\noindent Let $\mathbf{X}_{[1,\infty)}$ be an $HLS_{K}(\pi,\nu,\alpha_{1},\ldots,\alpha_{K-2})$ exchangeable sequence. For any $\mathbf{x}_{n}=(x_{1},\ldots,x_{n})\in\{d_{1},\ldots,d_{K}\}^{n}$, containing exactly $z_{i}$ coordinates equal to $d_{i}$, $i=1,\ldots,K-1$, setting $\theta:=\sum_{i=1}^{K-1}\theta_{i}$ and $z:=\sum_{i=1}^{K-1}z_{i}$, one has that
\begin{eqnarray}
\nonumber
&& \mathbb{P}(\mathbf{X}_{[n]}=\mathbf{x}_{n})\\
\nonumber
&&=\frac{1}{B(\pi,\nu)}\int_{\Theta_{K-1}}(1-\theta)^{n-z}(\prod_{i=1}^{K-1}\text{ }\theta_{i}^{z_{i}})\delta_{\alpha_{1}(1-\theta_{1})}(\text{d}\theta_{2})\cdots\delta_{\alpha_{K-2}(1-\theta_{1})}(\text{d}\theta_{K-1})\theta_{1}^{\pi-1}(1-\theta_{1})^{\nu-1}\text{d}\theta_{1}\\
&&=
\nonumber
\frac{(1-\alpha)^{n-z}\prod_{i=2}^{K-1}\alpha_{i-1}^{z_{i}}}{B(\pi,\nu)}\int_0^{1}\theta_{1}^{z_{1}+\pi-1}(1-\theta_{1})^{n-z_{1}+\nu-1}\text{d}\theta_{1}\\
&&=
\label{probHLS}
\left[(1-\alpha)^{n-z}\prod_{i=1}^{K-2}\alpha_{i}^{z_{i+1}}\right]\frac{B(z_{1}+\pi,n-z_{1}+\nu)}{B(\pi,\nu)}.
\end{eqnarray}

The following statement provides a negative answer to Problem A (as discussed in Section \ref{ss:ovw}), and is one of the main achievement of the present paper. In particular, it shows that a naive generalization of Theorem \ref{t:eldkp}-(II) cannot be achieved for sets containing strictly more than 2 elements.

\begin{theorem}\label{t:main} For any $K\geq 3$ and any choice of the parameters $\pi,\nu>0$ and $\alpha_{1},\ldots,\alpha_{K-2}>0$ with $\sum_{i=1}^{K-2}\alpha_{i}<1$, the corresponding $HLS_{K}(\pi,\nu,\alpha_{1},\ldots,\alpha_{K-2})$ sequence is Hoeffding decomposable while being neither i.i.d. nor P\'olya.
\end{theorem}

Section \ref{s:maintools} contains a combinatorial characterization of Hoeffding decomposability in the framework of exchangeable sequences taking values in a finite set with $K\geq 3$ elements. Such a result will be our main tool in the proof of Theorem \ref{t:main}, as detailed in the subsequent Section \ref{s:proof}.

\section{A combinatorial characterization of Hoeffding-decomposability on finite spaces}\label{s:maintools}

\subsection{Framework}
Let $\mathbf{X}_{[1,\infty)}:=\left\{ X_{n}:n\geq
1\right\} $ be a sequence of exchangeable random variables with
values in $D=\left\{ d_{1},...,d_{K}\right\},$ $K\geq 3$. Let $\gamma$ be the de
Finetti measure associated with $\mathbf{X}_{[1,\infty)}$.
Throughout this section, we will systematically assume that $\mathbf{X}_{[1,\infty)}$ is such that
\begin{equation}
\label{for:nd}
\mathbb{P}(\mathbf{X}_{[n]}=\mathbf{x}_n)>0.\qquad\forall\mathbf{x}_n\in D^n,\quad\forall n\geq1.
\end{equation}

\medskip

\noindent In the sequel, we will adopt the following notation: let $\mathcal{%
N}\left( n,K\right) $ denote the set of weak $K$-compositions of $n,$ that
is the collection of all vectors $\mathbf{i}_{K}=\left(
i_{1},...,i_{K}\right) \in \mathbb{N}^{K}$ such that $\sum_{j=1}^{K}i_{j}=n.$
For each $n\geq 1$ and each $\mathbf{i}_{K}\in \mathcal{N}\left( n,K\right)
, $ define the set%
\begin{equation*}
C\left( n,\mathbf{i}_{K}\right) := \left\{ \mathbf{x}_{n}\in
D^{n}:\sum_{h=1}^{n}\mathbf{1}\left( x_{h}=d_{j}\right) =i_{j},\text{ }%
j=1,...,K\right\} .
\end{equation*}%
By exchangeability of $\mathbf{X}_{\left[ 1,\infty \right) }$ and symmetry
of all $\varphi \in \Xi _{n}\left( \mathbf{X}_{\left[ 1,\infty \right)
}\right) $, for all $n\geq 2$ and all $\mathbf{i}_{K}\in \mathcal{N}\left(
n,K\right) ,$ the functions $\mathbf{x}_{n}\mapsto \mathbb{P}\left( \mathbf{X%
}_{\left[ n\right] }=\mathbf{x}_{n}\right) $ and $\mathbf{x}_{n}\mapsto
\varphi \left( \mathbf{x}_{n}\right) $ are constant on $C\left( n,\mathbf{i}%
_{K}\right) .$ The constant values taken by each of these functions will be
denoted, respectively, $\mathbb{P}_{n}\left( i_{1},...,i_{K-1}\right) $ and $%
\varphi _{n}\left( i_{1},...,i_{K-1}\right) .$ Note that the omission of the
last coordinate of the vector $\mathbf{i}_{K}$ comes from the fact that its
value is completely determined by those of the prevoius $K-1$ coordinates.

\begin{remark}[On multinomial coefficients]{\rm Consider integers $m\geq 1$ and $b_1,...,b_k \geq 0$ such that $\sum b_i \leq m$.  In what follows we adopt the notation
$ \binom{m}{b_1,\cdot \cdot \cdot , \text{ }b_k}$ in order to indicate the multinomial coefficient
$$
\frac{m!} {b_1! \cdots b_k ! (m - \sum b_i)!}.
$$
We shall also use the following special ``star notation'':
\begin{equation}
\label{multinomast}
\binom{m}{b_1,\cdot \cdot \cdot , b_k}_{\ast}={\binom{m}{b_1}}_{\ast} {\binom{m-b_1}{b_2}}_{\ast} \cdots {\binom{ m- (b_1+\cdots +b_{k-1} ) }{b_k}}_{\ast},
\end{equation}
where
\begin{equation}
\label{binomast}
{\binom{a}{b}}_{\ast}=\binom{a}{b}\mathbf{1}_{\{0,\ldots,a\}}(b),
\end{equation}
and $\binom{a}{b}$ is the usual binomial coefficient. Note that $ \binom{m}{b_1,\cdot \cdot \cdot , \text{ }b_k} =  \binom{m}{b_1,\cdot \cdot \cdot , \text{ }b_k}_\ast$, whenever the binomial coefficients on the RHS of (\ref{binomast}) are all different from zero.
}
\end{remark}

\subsection{Two technical lemmas}

Our first technical result concerns the structure of the spaces $\Xi _{n}\left( \mathbf{X}_{\left[ 1,\infty
\right) }\right) $ introduced in Section \ref{ss:wi}.

\begin{lemma}\label{lemmozzo}
If $\mathbf{X}_{\left[ 1,\infty
\right) }$ is an exchangeable random sequence satisfying (\ref{for:nd}), then the vector space $\Xi _{n}\left( \mathbf{X}_{\left[ 1,\infty
\right) }\right) $ \textit{is the }$\left[ \binom{n+K-1}{K-1}-\binom{n+K-2}{K-1}\right]$-dimensional vector space spanned by the symmetric
kernels $\varphi_{n}^{\mathbf{m}_{K-2}}$, $\mathbf{m}_{K-2}=(m_{1},\ldots,m_{K-2})\in\bigcup\limits_{a=0}^{n}\mathcal{N}(a,K-2)$, such that, for each $\mathbf{m}_{K-2}=(m_{1},\ldots,m_{K-2})\in\bigcup\limits_{a=0}^{n}\mathcal{N}(a,K-2)$, and each $\mathbf{i}_{K}=(i_{1},\ldots,i_{K-1},i_{K})\in\mathcal{N}(n,K)$,
\begin{equation}
\varphi _{n}^{\mathbf{m}_{K-2}}\left( i_{1},...,i_{K-1}\right) =\left(
-1\right) ^{i_{1}}\binom{i_{1}}{m_{1}-i_{2}\cdot \cdot \cdot \text{ }%
m_{K-2}-i_{K-1}}_{\ast }\frac{\mathbb{P}_{n}\left(
0,m_{1},...,m_{K-2}\right) }{\mathbb{P}_{n}\left( i_{1},...,i_{K-1}\right) }.
\label{labbbase}
\end{equation}
\end{lemma}

\begin{proof}
The fact that
$$\dim(\Xi_{n}(\mathbf{X}_{[1,\infty)}))=\binom{n+K-1}{K-1}-\binom{n+K-2}{K-1}$$
follows from Proposition 6 in El-Dakkak and Peccati [2008].
In order to prove the rest of the statement, we will show that the collection
$$\Phi_{n}:=\left\{\varphi_{n}^{\mathbf{m}_{K-2}}\colon \mathbf{m}_{K-2}\in\bigcup_{a=0}^{n}\mathcal{N}(a,K-2)\right\},$$
is indeed a basis of the vector space $\Xi_{n}(\mathbf{X}_{[1,\infty)})$. To do that, we will first show that, for each $\mathbf{i}_{K}\in\{\mathbf{i}_{K}=(i_{1},\ldots,i_{K})\in\mathcal{N}(n,K):i_{1}\geq1\}$, there exists a {\it linear} mapping $f_{\mathbf{i}_{K}}\colon \mathfrak{R}^A \to\mathfrak{R}$, where $A = \bigcup\limits_{a=0}^{n}\mathcal{N}(a,K-2)$, such that, for all $\varphi_{n}\in\Xi_{n}(\mathbf{X}_{[1,\infty)})$,
\begin{equation}
\label{rapresentasiun}
\varphi_{n}(i_{1},\ldots,i_{K-1})=f_{\mathbf{i}_{K}}\left(\varphi_{n}(0,\mathbf{m}_{K-2}) : \mathbf{m}_{K-2}\in\bigcup_{a=0}^{n}\mathcal{N}(a,K-2)\right).
\end{equation}
Once the explicit representation (\ref{rapresentasiun}) will be at hand, the characterization of $\Phi_{n}$ as a basis will be deduced from the fact that, for $\mathbf{m}_{K-2}$ in $\bigcup\limits_{a=0}^{n}\mathcal{N}(a,K-2)$, the functions $\varphi_{n}^{\mathbf{m}_{K-2}}$ appearing in the statement verify the relation
\begin{equation}
\label{caraterisasiun}
\varphi_{n}^{\mathbf{m}_{K-2}}(i_{1},\ldots,i_{K})=f_{\mathbf{i}_{K}}\left(\mathbf{1}_{\{\mathbf{m}_{K-2}\}}(\mathbf{y}_{K-2}) : \mathbf{y}_{K-2}\in\bigcup_{a=0}^{n}\mathcal{N}(a,K-2)\right).
\end{equation}

\smallskip

\noindent Let $\varphi _{n}\in \Xi _{n}\left( \mathbf{X}_{\left[
1,\infty \right) }\right) .$ It turns out that, for all $\mathbf{i}_{K}\in \{\mathbf{i}_{K}=(i_{1},\ldots,i_{K})\in\mathcal{N}\left( n,K\right):i_{1}\geq 1\}$,
\begin{equation}
\varphi _{n}\left( i_{1},...,i_{K-1}\right) =-\frac{1}{\mathbb{P}_{n}\left(
i_{1},...,i_{K-1}\right) }\sum_{j_{1}=2}^{K}\varphi _{n}\cdot \mathbb{P}%
_{n}\left( \mu _{1}^{j_{1}}\left( i_{1},...,i_{K-1}\right) \right) ,
\label{shift_operator}
\end{equation}%
where $\varphi _{n}\cdot \mathbb{P}_{n}\left( \cdot \right) :=
\varphi _{n}\left( \cdot \right) \mathbb{P}_{n}\left( \cdot \right) ,$
where, for $1\leq l<p\leq K-1,$%
\begin{equation*}
\mu _{l}^{p}\left( i_{1},...,i_{K-1}\right) := \left(
i_{1},...,i_{l-1},i_{l}-1,i_{l+1},...,i_{p-1},i_{p}+1,i_{p+1},...,i_{K-1}%
\right) ,
\end{equation*}%
and where for $1\leq l\leq K-1,$%
\begin{equation*}
\mu _{l}^{K}\left( i_{1},...,i_{K-1}\right) := \left(
i_{1},...,i_{l-1},i_{l}-1,i_{l+1},...,i_{K-1}\right) .
\end{equation*}%
Before proving formula (\ref{shift_operator}), we make some remarks and
give simple examples regarding our notation in order to better clarify
it. For $1\leq l<p\leq K-1,$ the action of the operator $\mu _{l}^{p}$
consists in substracting $1$ from the $l$-th coordinate of the vector $%
\left( i_{1},...,i_{K-1}\right) $ an adding $1$ to the $p$-th coordinate.
For example:%
\begin{equation*}
\mu _{2}^{4}\left( 2,7,5,9,4\right) =\left( 2,6,5,10,4\right) .
\end{equation*}%
On the other hand, when $1\leq l\leq K-1$ and $p=K,$ the action of the
operator $\mu _{l}^{K}$ consists in just substracting $1$ from the $l$-th
coordinate. This is consistent with our conventions since we are ommitting
the $K$-th coordinate of the vectors $\left( i_{1},...,i_{K}\right) $; in
other words the the $1$ substracted from the $l$-th coordinate is actually
added the the last coordinate whose value we are ommitting since it is
completely determined by the values of the previous ones.

\medskip

\noindent We shall now prove formula (\ref{shift_operator}). Fix $%
n\geq 2$ and $\varphi _{n}\in \Xi _{n}\left( \mathbf{X}_{\left[ 1,\infty
\right) }\right) .$ By the definition of $\Xi _{n}\left( \mathbf{X}_{\left[
1,\infty \right) }\right) ,$ we must have%
\begin{equation*}
E\left( \varphi _{n}\left( \mathbf{X}_{\left[ n\right] }\right) \text{ }|%
\text{ }\mathbf{X}_{\left[ 2,n\right] }\right) =0.
\end{equation*}%
Then, for any arbitrarily fixed $\mathbf{x}_{n-1}\in D^{n-1},$%
\begin{equation*}
\sum_{i=1}^{K}\varphi _{n}\left( d_{i},x_{2},...,x_{n-1}\right) \frac{%
\mathbb{P}_{n}\left( X_{1}=d_{i},\mathbf{X}_{\left[ 2,n\right] }=\mathbf{x}%
_{n-1}\right) }{\mathbb{P}_{n-1}\left( \mathbf{X}_{\left[ 2,n\right] }=%
\mathbf{x}_{n-1}\right) }=0.
\end{equation*}%
Suppose $\mathbf{x}_{n-1}\in C\left( n-1,\mathbf{h}_{K}\right) ,$ for some $%
\mathbf{h}_{K}\in \mathcal{N}\left( n-1,K\right) .$ Then, by (\ref{for:nd}), the just-stated
formula is equivalent to%
\begin{equation*}
\varphi _{n}\cdot \mathbb{P}_{n}\left( h_{1}+1,h_{2},...,h_{K-1}\right)
+\sum_{j_{1}=2}^{K}\varphi _{n}\cdot \mathbb{P}_{n}\left( \mu
_{1}^{j_{1}}\left( h_{1}+1,h_{2},...,h_{K-1}\right) \right) =0,
\end{equation*}%
thus proving that (\ref{shift_operator}) holds for the vector $\mathbf{i}_{K}^{*}=\left( i_{1},...,i_{K}\right) := \left(
h_{1}+1,h_{2},...,h_{K}\right) $. Clearly, $\mathbf{i}_{K}^{*}\in\{\mathbf{i}_{K}=(i_{1},\ldots,i_{K})\in\mathcal{N}\left( n,K\right):i_{1}\geq 1\}$. To see that (\ref{shift_operator}) holds for all $\mathbf{i}_{K}\in \{\mathbf{i}_{K}=(i_{1},\ldots,i_{K})\in\mathcal{N}\left(
n,K\right):i_{1}\geq 1\}$, observe that for any such $\mathbf{i}_{K},$
there exists $\mathbf{h}_{K}\in \mathcal{N}\left( n-1,K\right) $ such that $%
\mathbf{i}_{K}=\left( h_{1}+1,...,h_{K}\right) $. This proves also that%
\begin{equation}
card\left( \left\{ \mathbf{i}_{K}=\left( i_{1},...,i_{K}\right) \in \mathcal{%
N}\left( n,K\right) :i_{1}\geq 1\right\} \right) =card\left( \mathcal{N}%
\left( n-1,K\right) \right) .  \label{quandebello}
\end{equation}

\medskip

\noindent Now, recursion in (\ref{shift_operator}) gives%
\begin{equation}
\varphi _{n}\left( i_{1},...,i_{K-1}\right) =\frac{\left( -1\right) ^{i_{1}}%
}{\mathbb{P}_{n}\left( i_{1},...,i_{K-1}\right) }\sum_{j_{1}=2}^{K}\cdot
\cdot \cdot \sum_{j_{i_{1}}=2}^{K}\varphi _{n}\cdot \mathbb{P}_{n}\left( \mu
_{1}^{j_{i_{1}}}\circ \cdot \cdot \cdot \circ \text{ }\mu _{1}^{j_{1}}\left(
i_{1},...,i_{K-1}\right) \right) ,  \label{recursion}
\end{equation}%
where the operator $\mu _{1}^{j_{i_{1}}}\circ \cdot \cdot \cdot \circ $ $\mu
_{1}^{j_{1}}$ denotes the successive iteration of operators $\mu
_{1}^{j_{1}},...,\mu _{1}^{j_{i_{1}}}.$ For example,%
\begin{equation*}
\mu _{1}^{2}\circ \mu _{1}^{3}\circ \mu _{1}^{2}\circ \mu _{1}^{4}\left(
4,7,5,4,9\right) =\mu _{1}^{2}\left( \mu _{1}^{3}\left( \mu _{1}^{2}\left(
\mu _{1}^{4}\left( 4,7,5,4,9\right) \right) \right) \right) =\left(
0,9,6,5,9\right) .
\end{equation*}%
To see this, fix $j_{1}\in \left\{ 2,...,K\right\} $ and apply (\ref%
{shift_operator}) to $\varphi _{n}\left( \mu _{1}^{j_{1}}\left(
i_{1},...,i_{K-1}\right) \right) $ to obtain%
\begin{equation*}
\varphi _{n}\left( \mu _{1}^{j_{1}}\left( i_{1},...,i_{K-1}\right) \right) =-%
\frac{1}{\mathbb{P}_{n}\left( \mu _{1}^{j_{1}}\left(
i_{1},...,i_{K-1}\right) \right) }\sum_{j_{2}=2}^{K}\varphi _{n}\cdot
\mathbb{P}_{n}\left( \mu _{1}^{j_{2}}\circ \mu _{1}^{j_{1}}\left(
i_{1},...,i_{K-1}\right) \right) .
\end{equation*}%
Do that for all $j_{1}\in \left\{ 2,...,K\right\} $ and plug in (\ref%
{shift_operator}) to obtain%
\begin{equation*}
\varphi _{n}\left( i_{1},...,i_{K-1}\right) =\frac{1}{\mathbb{P}_{n}\left(
i_{1},...,i_{K-1}\right) }\sum_{j_{1}=1}^{K}\sum_{j_{2}=2}^{K}\varphi
_{n}\cdot \mathbb{P}_{n}\left( \mu _{1}^{j_{2}}\circ \mu _{1}^{j_{1}}\left(
i_{1},...,i_{K-1}\right) \right) .
\end{equation*}%
Iterating the process $i_{1}$ times gives (\ref{recursion}).

\medskip

\noindent Next, observe that the term $\varphi _{n}\cdot \mathbb{P}%
_{n}\left( \mu _{1}^{j_{i_{1}}}\circ \cdot \cdot \cdot \circ \text{ }\mu
_{1}^{j_{1}}\left( i_{1},...,i_{K-1}\right) \right) $ is certainly of the
form%
\begin{equation}
\varphi _{n}\cdot \mathbb{P}_{n}\left(
0,i_{2}+b_{1},i_{3}+b_{2},....,i_{K-1}+b_{K-2}\right) ,
\label{termine_generico}
\end{equation}%
where $b_{v}=\sum_{t=1}^{i_{1}}\mathbf{1}\left( j_{t}=v+1\right) ,$ $%
v=1,...,K-2$ is the number of $1$'s substracted from the first coordinate of
$\left( i_{1},...,i_{K-1}\right) $ and added to coordinate $v+1.$ Note that $%
b_{K-1}$ (i.e. the number of $1$'s substracted from the first coordinate and
added to the last) is completely determined by the vector $\mathbf{b}%
_{K-2}=\left( b_{1},...,b_{K-2}\right) .$ Clearly, the operator $\mu
_{1}^{j_{i_{1}}}\circ \cdot \cdot \cdot \circ $ $\mu _{1}^{j_{1}}$ is
commutative in the sense that%
\begin{equation*}
\mu _{1}^{j_{i_{1}}}\circ \cdot \cdot \cdot \circ \text{ }\mu
_{1}^{j_{1}}=\mu _{1}^{j_{\sigma \left( i_{1}\right) }}\circ \cdot \cdot
\cdot \circ \text{ }\mu _{1}^{j_{\sigma \left( 1\right) }},
\end{equation*}%
for any permutation $\sigma$ of $\left( 1,...,i_{1}\right) .$ It follows that, for
any fixed $\left( b_{1},...,b_{K-2}\right) \in \bigcup\limits_{a=0}^{i_{1}}%
\mathcal{N}\left( a,K-2\right) $ the term (\ref{termine_generico}) occurs
exactly $\binom{i_{1}}{b_{1},b_{2}...,b_{K-2}}$ times in the sum described
in (\ref{recursion}). Consequently, (\ref{recursion}) can be rewritten as
follows:%
\begin{equation}
\varphi _{n}\left( i_{1},...,i_{K-1}\right) =\frac{\left( -1\right) ^{i_{1}}%
}{\mathbb{P}_{n}\left( i_{1},...,i_{K-1}\right) }\sum_{(+)}\binom{i_{1}}{b_{1},...,b_{K-2}}\varphi _{n}\cdot \mathbb{P}%
_{n}\left( 0,i_{2}+b_{1},i_{3}+b_{2}....,i_{K-1}+b_{K-2}\right) ,
\label{orribile}
\end{equation}
where the sum $(+)$ is extended to all vectors $\mathbf{b}_{K-2}=(b_{1},\ldots,b_{K-2})\in\bigcup\limits_{a=0}^{i_{1}}\mathcal{N}(a,K-2)$. Set $m_{p}:=b_{p}+i_{p+1}$, $p=1,\ldots,K-2$, and rewrite (\ref{orribile}) as
\begin{equation}
\label{orribileinemme}
\varphi_{n}(i_{1},\ldots,i_{K-1})=\frac{(-1)^{i_{1}}}{\mathbb{P}_{n}(i_{1},\ldots,i_{K-1})}\sum_{(\sharp)}\binom{i_{1}}{m_{1}-i_{2},\ldots,m_{K-2}-i_{K-1}}\varphi_{n}\cdot\mathbb{P}_{n}(0,m_{1},\ldots,m_{K-2}),
\end{equation}
where the sum $(\sharp)$ is extended to all vectors $(m_{1},\ldots,m_{K-2})\in\bigcup\limits_{a=0}^{n}\mathcal{N}(a,K-2)$, such that

\begin{eqnarray*}
m_{1}&\in&\{i_{2},\ldots,i_{1}+i_{2}\},\\
m_{2}&\in&\{i_{3},\ldots,(i_{1}-m_{1})+i_{2}+i_{3}\}\\
m_{3}&\in&\{i_{4},\ldots,(i_{1}-m_{1})+(i_{2}-m_{2})+i_{3}+i_{4}\}\\
&\vdots&\\
m_{K-2}&\in&\left\{i_{K-1},\ldots,\sum_{v=1}^{K-3}(i_{v}-m_{v})+i_{K-2}+i_{K-1}\right\}.
\end{eqnarray*}

\noindent It is immediately seen that the multinomial coefficient in (\ref{orribileinemme}) is always well defined. It follows that (\ref{orribileinemme}) can be rewritten, using the convention defined in (\ref{multinomast}), as

\begin{equation}
\label{orribileasterisco}
\varphi_{n}(i_{1},\ldots,i_{K-1})=\frac{(-1)^{i_{1}}}{\mathbb{P}_{n}(i_{1},\ldots,i_{K-1})}\sum_{(=)}\binom{i_{1}}{m_{1}-i_{2},\ldots,m_{K-2}-i_{K-1}}_{\ast}\varphi_{n}\cdot\mathbb{P}_{n}(0,m_{1},\ldots,m_{K-2}),
\end{equation}
where the sum $(=)$ is extended to all vectors $\mathbf{m}_{K-2}=(m_{1},\ldots,m_{K-2})\in\bigcup\limits_{a=0}^{n}\mathcal{N}(a,K-2)$.
Since equality (\ref{orribileasterisco}) holds for any $\varphi_{n}\in\Xi_{n}(\mathbf{X}_{[1,\infty)})$ and any $\mathbf{i}_{K}\in\{\mathbf{i}_{K}=(i_{1},\ldots,i_{k})\in\mathcal{N}(n,K):i\geq1\}$, then we have proved (\ref{rapresentasiun}). The claim of the present step of the proof follows, now, immediately from (\ref{caraterisasiun}). \hfill
\end{proof}

Next, adapting the arguments rehearsed in the
proof of Lemma 3 of El-Dakkak and Peccati $[2008]$ yields the
following statement about symmetrizations.

\begin{lemma}
\label{lemma3_K}\textit{Fix }$m\geq 2$ \textit{and }$v\in \left\{
1,...,m-1\right\} $\textit{\ and let the application}%
\begin{equation*}
f_{v,m-v}:D^{m}\mapsto \mathfrak{R}:\left( x_{1},...,x_{m}\right) \mapsto
f_{v,m-v}\left( x_{1},...,x_{m}\right) ,
\end{equation*}%
\textit{where }$D=\left\{ d_{1},...,d_{K}\right\} ,$ \textit{be separately
symmetric in the variables }$\left( x_{1},...,x_{v}\right) $ \textit{and }$%
\left( x_{v+1},...,x_{m}\right) $ \textit{(and not necessarily as a function
on }$D^{m}$\textit{). Then, for any }$\mathbf{x}_{m}\in C\left( m,\mathbf{z}%
_{K}\right) $ \textit{for some }$\mathbf{z}_{K}=\left(
z_{1},...,z_{K}\right) \in \mathcal{N}\left( m,K\right) ,$ \textit{the
canonical symmetrization of }$f,$ \textit{denoted }$\tilde{f},$ \textit{%
reduces to}%
\begin{equation}
\tilde{f}\left( \mathbf{x}_{m}\right) =\frac{\sum_{\left( \ast \right) }%
\binom{v}{k_{1},k_{2},...,k_{K-1}}\binom{m-v}{%
z_{1}-k_{1},z_{2}-k_{2},...,z_{K-1}-k_{K-1}}f_{v,m-v}\left( \left(
k_{1},...,k_{K-1}\right) ,\left( z_{1}-k_{1},...,z_{K}-k_{K}\right) \right)
}{\sum_{\left( \ast \right) }\binom{v}{k_{1},k_{2},...,k_{K-1}}\binom{m-v}{%
z_{1}-k_{1},z_{2}-k_{2},...,z_{K-1}-k_{K-1}}},  \label{quindici_K}
\end{equation}%
\textit{where the sums }$\left( \ast \right) $ are extended to all the
vectors $\left( k_{1},...,k_{K-1}\right) $ with%
\begin{eqnarray*}
k_{1} &\in &\left\{ 0\vee \left[ z_{1}-\left( m-v\right) \right]
,...,z_{1}\wedge v\right\} , \\
k_{2} &\in &\left\{ 0\vee \left[ z_{2}-\left( m-v\right) -\left(
z_{1}-k_{1}\right) \right] ,...,z_{2}\wedge \left( v-k_{1}\right) \right\} ,
\\
&&\vdots \\
k_{K-1} &\in &\left\{ 0\vee \left[ z_{K-1}-\left( m-v\right) -\Sigma
_{1}^{K-2}\left( z_{p}-k_{p}\right) \right] ,...,z_{K-1}\wedge \left(
v-\Sigma _{1}^{K-2}k_{p}\right) \right\}
\end{eqnarray*}%
\textit{\ (the set of all such vectors will, in the sequel, be referred to
as the set of all }$\left( m,v,\mathbf{z}_{K}\right) $-coherent vectors)
\textit{and where }$f_{v,m-v}\left( \left( k_{1},...,k_{K-1}\right) ,\left(
z_{1}-k_{1},...,z_{K}-k_{K}\right) \right) $ \textit{denotes the common
value of }$f_{v,m-v}\left( \mathbf{y}_{m}\right) $ \textit{when }$\mathbf{y}%
_{m}=\left( y_{1},...,y_{m}\right) $ \textit{is} \textit{such that the
subvector} $\left( y_{1},...,y_{v}\right) $ \textit{contains exactly} $k_{i}$
\textit{coordinates equal to} $d_{i},$ $i=1,...,K,$ \textit{and the subvector%
} $\left( y_{v+1},...,y_{m}\right) $ \textit{contains exactly} $\left(
z_{i}-k_{i}\right) $ \textit{coordinates equal to} $d_{i},$ $i=1,...,K.$

As a consequence, $\tilde{f}_{v,m-v}\left( \mathbf{x}_{m}\right) =0$ for
every $\mathbf{x}_{m}\in D^{m}$ if, and only if, for all $\mathbf{z}%
_{K}=\left( z_{1},...,z_{K}\right) \in \mathcal{N}\left( m,K\right) ,$%
\begin{eqnarray}\notag
&& \sum_{\left( \ast \right) }\binom{v}{k_{1},k_{2},...,k_{K-1}}\binom{m-v}{%
z_{1}-k_{1},z_{2}-k_{2},...,z_{K-1}-k_{K-1}}\times \\
&&\quad\quad\quad\quad\quad\quad\quad\times f_{v,m-v}\left( \left(
k_{1},...,k_{K-1}\right) ,\left( z_{1}-k_{1},...,z_{K}-k_{K}\right) \right)
=0.  \label{sedici_K}
\end{eqnarray}
\end{lemma}

\subsection{The characterization}

We are now ready to prove the announced full
characterization of $D$-valued Hoeffding decomposable exchangeable
sequences satisfying (\ref{for:nd}), where $D=\left\{ d_{1},...,d_{K}\right\} $. To this end, recall
that, for every symmetric $\varphi :D^{n}\rightarrow \mathfrak{R},$ every $%
u=2,...,n$ and every $\mathbf{x}_{n-1}=\left( x_{1},...,x_{n-1}\right) \in
D^{n-1},$%
\begin{equation*}
\left[ \varphi \right] _{n,n-1}^{\left( n-u\right) }\left( \mathbf{x}%
_{n-1}\right) =\mathbb{E}\left( \varphi \left( \mathbf{X}_{\left[ n\right]
}\right) \text{ }|\text{ }\mathbf{X}_{\left[ u+1,u+n-1\right] }=\mathbf{x}%
_{n-1}\right) .
\end{equation*}%
Observe that the function $\left[ \varphi \right] _{n,n-1}^{\left(
n-u\right) }:D^{n-1}\rightarrow \mathfrak{R}$ clearly meets the symmetry
properties of Lemma \ref{lemma3_K} with $m=n-1$ and $v=n-u.$ Now, fix $%
\mathbf{z}_{K}=\left( z_{1},...,z_{K}\right) \in \mathcal{N}\left(
n-1,K\right) $ and suppose $\mathbf{x}_{n-1}=\left( x_{1},...,x_{n-1}\right)
\in C\left( n-1,\mathbf{z}_{K}\right) $ is such that $\sum_{t=1}^{n-u}%
\mathbf{1}\left( x_{t}=d_{p}\right) =k_{p},$ $p=1,...K-1.$ Then

\begin{eqnarray}
\notag
\left[ \varphi \right] _{n,n-1}^{\left( n-u\right) }\left( \mathbf{x}%
_{n-1}\right)& =& \sum_{q_{1}=0}^{u}\sum_{q_{2}=0}^{u-q_{1}}\cdot \cdot \cdot
\sum_{q_{K-1}=1}^{u-\Sigma _{1}^{K-2}q_{j}}\binom{u}{q_{1},...,q_{K-1}}
\varphi _{n}\left( k_{1}+q_{1},...,k_{K-1}+q_{K-1}\right) \times
 \\
&&\quad\quad\quad \quad \quad \times  \frac{\mathbb{P}_{n+u-1}\left( z_{1}+q_{1},...,z_{K-1}+q_{K-1}\right)}{\mathbb{P}
_{n-1}\left( z_{1},...,z_{K-1}\right) }.  \label{diciassette_K}
\end{eqnarray}%
By applying (\ref{sedici_K}) in the case $m=n-1,$ we deduce that $\widetilde{%
\left[ \varphi \right] }_{n,n-1}^{\left( n-u\right) }\left( \mathbf{x}%
_{n-1}\right) =0$ if and only if
\begin{eqnarray}\notag
&& \sum_{\left( \ast \right) }\binom{n-u}{k_{1},k_{2},...,k_{K-1}}\binom{u-1}{%
z_{1}-k_{1},z_{2}-k_{2},...,z_{K-1}-k_{K-1}}\times  \\
&&\quad\quad\quad \times \left[ \varphi \right]
_{n,n-1}^{\left( n-u\right) }\left( \left( k_{1},...,k_{K-1}\right) ,\left(
z_{1}-k_{1},...,z_{K}-k_{K}\right) \right) =0,  \label{diciotto_K}
\end{eqnarray}%
where the sum $\left( \ast \right) $ is extended to all $\left( n-1,n-u,%
\mathbf{z}_{K}\right) $-coherent vectors $\left( k_{1},...,k_{K-1}\right) ,$
i.e.%
\begin{eqnarray*}
k_{1} &\in &\left\{ 0\vee \left[ z_{1}-\left( u-1\right) \right]
,...,z_{1}\wedge \left( n-u\right) \right\} , \\
k_{2} &\in &\left\{ 0\vee \left[ z_{2}-\left( u-1\right) -\left(
z_{1}-k_{1}\right) \right] ,...,z_{2}\wedge \left( n-u-k_{1}\right) \right\}
, \\
&&\vdots \\
k_{K-1} &\in &\left\{ 0\vee \left[ z_{K-1}-\left( u-1\right) -\Sigma
_{1}^{K-2}\left( z_{p}-k_{p}\right) \right] ,...,z_{K-1}\wedge \left(
n-u-\Sigma _{1}^{K-2}k_{p}\right) \right\} ,
\end{eqnarray*}%
and where the notation $\left[ \varphi \right] _{n,n-1}^{\left( n-u\right)
}\left( \left( k_{1},...,k_{K-1}\right) ,\left(
z_{1}-k_{1},...,z_{K}-k_{K}\right) \right) $ has been introduced to indicate
the common value of $\left[ \varphi \right] _{n,n-1}^{\left( n-u\right)
}\left( \mathbf{x}_{n-1}\right) ,$ for all $\mathbf{x}_{n-1}\left(
x_{1},...,x_{n-1}\right) \in C\left( n-1,\mathbf{z}_{K}\right) $ such that $%
\sum_{t=1}^{n-u}\mathbf{1}\left( x_{t}=d_{p}\right) =k_{p},$ $p=1,...,K-1.$

Now recall that, by Theorem \ref{t:eldkp}-(I),
Hoeffding decomposability and weak independence are equivalent provided
condition (\ref{theorem0}) is verified. The fact that such condition is
verified in our case is a consequence of Point 2 of Proposition 6 in Peccati
and El-Dakkak $\left[ 2008\right] .$ Moreover, $\mathbf{X}_{\left[ 1,\infty
\right) }$ is weakly independent if, and only if, for all $n\geq 2$ and all $%
\varphi _{n}\in \Xi _{n}\left( \mathbf{X}_{\left[ 1,\infty \right) }\right)
, $ one has $\varphi _{n}\in \tilde{\Xi}_{n,n-u}\left( \mathbf{X}_{\left[
1,\infty \right) }\right) ,$ for all $u=2,...,n.$ By Lemma \ref{lemmozzo}, this implies that for every $\mathbf{m}_{K-2}\in \bigcup\limits_{a=0}^{n}%
\mathcal{N}\left( a,K-2\right) ,$ the corresponding basis function $\varphi
_{n}^{\mathbf{m}_{K-2}}$ belongs to $\tilde{\Xi}_{n,n-u}\left( \mathbf{X}_{%
\left[ 1,\infty \right) }\right) .$ On the other hand since any $\varphi
_{n}\in \Xi _{n}\left( \mathbf{X}_{\left[ 1,\infty \right) }\right) $ is a
linear combination of the basis functions $\varphi _{n}^{\mathbf{m}_{K-2}},$
we deduce that weak independence occurs if, and only if, for all $\mathbf{m}%
_{K-2}\in \bigcup\limits_{a=0}^{n}\mathcal{N}\left( a,K-2\right) ,$ $%
\varphi _{n}^{\mathbf{m}_{K-2}}\in \tilde{\Xi}_{n,n-u}\left( \mathbf{X}_{%
\left[ 1,\infty \right) }\right) .$ In other words, weak independence occurs
if, and only if, for all $n\geq 2,$ all $u=2,...,n$, all $\mathbf{z}%
_{K}=\left( z_{1},...,z_{K}\right) \in \mathcal{N}\left( n-1,K\right) ,$ and
all $\mathbf{m}_{K-2}\in \bigcup\limits_{a=0}^{n}\mathcal{N}\left(
a,K-2\right) $%
\begin{eqnarray}\notag
&& \sum_{\left( \ast \right) }\binom{n-u}{k_{1},k_{2},...,k_{K-1}}\binom{u-1}{%
z_{1}-k_{1},z_{2}-k_{2},...,z_{K-1}-k_{K-1}}\times \\
&& \quad\quad\quad \times \left[ \varphi _{n}^{\mathbf{m}%
_{K-2}}\right] _{n,n-1}^{\left( n-u\right) }\left( \left(
k_{1},...,k_{K-1}\right) ,\left( z_{1}-k_{1},...,z_{K}-k_{K}\right) \right)
=0,  \label{diciannove_K}
\end{eqnarray}%
the sum $\left( \ast \right) $ being, as usual, extended to all $\left(
n-1,n-u,\mathbf{z}_{K}\right) $-coherent vectors $\left(
k_{1},...,k_{K-1}\right) .$ Substituting (\ref{labbbase}) and (\ref%
{diciassette_K}) in (\ref{diciotto_K}), one has that (\ref{diciannove_K}) is
true if, and only if, all $\mathbf{m}_{K-2}\in \bigcup\limits_{a=0}^{n}%
\mathcal{N}\left( a,K-2\right) ,$ for the following quantity equals $0$:%
\begin{eqnarray}
&&\frac{\mathbb{P}_{n}\left( 0,m_{1},...,m_{K-2}\right) }{\mathbb{P}%
_{n-1}\left( z_{1},...,z_{K-1}\right) }\sum_{\left( \ast \right) }\left(
-1\right) ^{k_{1}}\binom{n-u}{k_{1},k_{2},...,k_{K-1}}\binom{u-1}{%
z_{1}-k_{1},z_{2}-k_{2},...,z_{K-1}-k_{K-1}}\times  \notag \\
&&\times \sum_{\left( \ast \ast \right) }\left( -1\right) ^{q_{1}}%
\binom{u}{q_{1},...,q_{K-1}} \label{venti_K} \\
&&\qquad\quad \binom{k_{1}+q_{1}}{m_{1}-k_{2}-q_{2}\cdot \cdot
\cdot \text{ }m_{K-2}-k_{K-1}-q_{K-1}}_{\ast }\frac{\mathbb{P}_{n+u-1}\left(
z_{1}+q_{1},...,z_{K-1}+q_{K-1}\right) }{\mathbb{P}_{n}\left(
k_{1}+q_{1},...,z_{K-1}+k_{K-1}\right) },  \notag
\end{eqnarray}%
where the sum $\left( \ast \right) $ is extended to all $\left( n-1,n-u,%
\mathbf{z}_{K}\right) $-coherent vectors $\left( k_{1},...,k_{K-1}\right) $
and the sum $\left( \ast \ast \right) $ is extended to all $\mathbf{q}%
_{K-1}=\left( q_{1},...,q_{K-1}\right) \in \bigcup\limits_{a=0}^{u}\mathcal{N%
}\left( a,K-1\right) .$ Note that%
\begin{eqnarray}
&& \frac{\mathbb{P}_{n+u-1}\left( z_{1}+q_{1},...,z_{K-1}+q_{K-1}\right) }{%
\mathbb{P}_{n}\left( k_{1}+q_{1},...,z_{K-1}+k_{K-1}\right) } \label{ventuno_K}\\
&& =\frac{1}{%
\binom{u-1}{z_{1}-k_{1},z_{2}-k_{2},...,z_{K-1}-k_{K-1}}}\mathbb{P}%
_{n+u-1}^{n}\left( z_{1}+q_{1},...,z_{K-1}+q_{K-1}\text{ }|\text{ }%
k_{1}+q_{1},...,z_{K-1}+k_{K-1}\right) , \notag
\end{eqnarray}%
where $\mathbb{P}_{n+u-1}^{n}\left( z_{1}+q_{1},...,z_{K-1}+q_{K-1}\text{ }|%
\text{ }k_{1}+q_{1},...,z_{K-1}+k_{K-1}\right) $ denotes the conditional
probability that the vector $\mathbf{X}_{\left[ n+u-1\right] }$ contains
exactly $z_{p}+q_{p}$ coordinates equal to $d_{p},$ $p=1,...,K-1,$ given
that the subvector $\mathbf{X}_{\left[ n\right] }$ contains exactly $%
k_{p}+q_{p}$ coordinates equal to $d_{p},$ $p=1,...,K-1.$ Plugging (\ref%
{ventuno_K}) into (\ref{venti_K}), one deduces immediately the announced characterization of
weak independence.

\begin{theorem}
\label{omaraus_K}Let $\mathbf{X}_{\left[ 1,\infty \right) }$ be an infinite
sequence of exchangeable $D$-valued random variables satisfying (\ref{for:nd}). For the sequence to be Hoeffding-decomposable, it is necessary and
sufficient that, for every $n\geq 2,$ every $u=2,...,n$, every $\mathbf{z}%
_{K}=\left( z_{1},...,z_{K}\right) \in \mathcal{N}\left( n-1,K\right) $ and
every $\mathbf{m}_{K-2}\in \bigcup\limits_{a=0}^{n}\mathcal{N}\left(
a,K-2\right) ,$ the following quantity equals $0$:%
\begin{eqnarray}
&&\sum_{\left( \ast \right) }\left( -1\right) ^{k_{1}}\binom{n-u}{%
k_{1},k_{2},...,k_{K-1}}\times  \label{ventidue_K} \\
&&\qquad \times \sum_{\left( \ast \ast \right) }\left( -1\right) ^{q_{1}}%
\binom{u}{q_{1},...,q_{K-1}}\binom{k_{1}+q_{1}}{m_{1}-k_{2}-q_{2}\cdot \cdot
\cdot \text{ }m_{K-2}-k_{K-1}-q_{K-1}}_{\ast }\times  \notag \\
&&  \qquad \qquad \qquad \times \mathbb{P}%
_{n+u-1}^{n}\left( z_{1}+q_{1},...,z_{K-1}+q_{K-1}\text{ }|\text{ }%
k_{1}+q_{1},...,q_{K-1}+k_{K-1}\right) ,\notag
\end{eqnarray}%
where the sum $\left( \ast \right) $ is extended to all $\left( n-1,n-u,%
\mathbf{z}_{K}\right) $-coherent vectors $\left( k_{1},...,k_{K-1}\right) ,$
and the sum $\left( \ast \ast \right) $ is extended to all $\mathbf{q}%
_{K-1}=\left( q_{1},...,q_{K-1}\right) \in \bigcup\limits_{a=0}^{u}\mathcal{N%
}\left( a,K-1\right) .$
\end{theorem}

In the next section, we will use the content of Theorem \ref{omaraus_K} in the case $K=3$.

\section{Proof of Theorem \ref{t:main}}\label{s:proof}

We start by stating a result that is easily deduced from the proof of the Theorem 1 in El-Dakkak and Peccati $[2008]$.
\begin{lemma}
\label{sommedentro}
For all $\pi,\nu>0$, all $n\geq 2$, all $u=2,\ldots,n$, all $z=0,\ldots,n-1$ and all $k\in\{\max\{0,z-(u-1)\},\ldots,\min\{z,n-u\}\}\},$
$$\sum_{q=0}^{u}(-1)^{q}\binom{u}{q}\frac{B(\pi+z+q,\nu+n+u-1-z-q)}{B(\pi+k+q,\nu+n-k-q)}=0.$$
\end{lemma}
To keep the notation as understandable as possible, we shall restrict ourselves to the case $K=3$. The proof for general $K$ carries out exactly in the same way. Let $\mathbf{X}_{[1,\infty)}$ be a $HLS_{3}(\pi,\nu,\alpha)$- exchangeable sequence with values in $D=\{d_{1},d_{2},d_{3}\}$, where $\pi,\nu>0$ and $0<\alpha<1$. By (\ref{probHLS}), the following facts are in order : (a) $\mathbf{X}_{[1,\infty)}$ satisfies (\ref{for:nd}), (b) $\mathbf{X}_{[1,\infty)}$ is neither i.i.d. nor a $K$-color P\'{o}lya sequence and (c) if $\mathbf{x}_{n}\in D^{n}$ contains exactly $z_{1}$ coordinates equal to $d_{1}$ and $z_{2}$ coordinates equal to $d_{2}$, then, in the language of the present paper,
\begin{equation}
\label{probHLS3}
\mathbb{P}(\mathbf{X}_{[n]}=\mathbf{x}_{n})=\mathbb{P}_{n}(z_{1},z_{2})=\alpha^{z_{2}}(1-\alpha)^{n-z_{1}-z_{2}}\frac{B(z_{1}+\pi,n-z_{1}+\nu)}{B(\pi,\nu)}.
\end{equation}
Recall that, by Theorem \ref{omaraus_K}, an exchangeable sequence with values in $D=\{d_{1},d_{2},d_{3}\}$ is Hoeffding- decomposable if and only if, for all $n\geq 2$, all $u=2,\ldots,n$, all $m=0,\ldots,n$ and all $(z_{1},z_{2})\in S(z_{1},z_{2}):=\{(z_{1},z_{2})\in\{0,\ldots,n-1\}^{2}\colon z_{1}+z_{2}\leq n-1\}$, one has

\begin{eqnarray*}
0=\sum_{k_{1}\in A(z_{1},u)}\sum_{k_{2}\in A_{z_{1},k_{1}}(z_{2},u)}\!\!\!\!\!&& (-1)^{k_{1}}\binom{n-u}{k_{1}\hspace{0.2cm}k_{2}}\binom{u-1}{z_{1}-k_{1}\hspace{0.2cm}z_{2}-k_{2}}\times\\
&&\times\sum_{q_{1}=0}^{u}\sum_{q_{2}=0}^{u-q_{1}}(-1)^{q_{1}}\binom{u}{q_{1}\hspace{0.2cm}q_{2}}\binom{k_{1}+q_{1}}{m-k_{2}-q_{2}}_{\ast}\frac{\mathbb{P}_{n+u-1}(z_{1}+q_{1},z_{2}+q_{2})}{\mathbb{P}_{n}(k_{1}+q_{1},k_{2}+q_{2})},
\end{eqnarray*}
where
$$A(z_{1},u)=\{\max\{0,z_{1}-(u-1)\},\ldots,\min\{z_{1},n-u\}\},$$
and
$$A_{z_{1},k_{1}}(z_{2},u)=\{\max\{0,z_{2}-(u-1)-(z_{1}-k_{1})\},\ldots,\min\{z_{1},n-u-k_{1}\}\},$$
and where the notation $\binom{a}{b}_{\ast}$ is defined in (\ref{binomast}).

\noindent It follows that $\mathbf{X}_{[1,\infty)}$ is Hoeffding-decomposable if and only if for all $n\geq 2$, all $u=2,\ldots,n$, all $m=0,\ldots,n$ and all $(z_{1},z_{2})\in S(z_{1},z_{2}):=\{(z_{1},z_{2})\in\{0,\ldots,n-1\}^{2}\colon z_{1}+z_{2}\leq n-1\}$, one has

\begin{eqnarray*}
0&=&\sum_{k_{1}\in A(z_{1},u)}\sum_{k_{2}\in A_{z_{1},k_{1}}(z_{2},u)}(-1)^{k_{1}}\binom{n-u}{k_{1}\hspace{0.2cm}k_{2}}\binom{u-1}{z_{1}-k_{1}\hspace{0.2cm}z_{2}-k_{2}}\times\\
&&\times C\sum_{q_{1}=0}^{u}\sum_{q_{2}=0}^{u-q_{1}}(-1)^{q_{1}}\binom{u}{q_{1}\hspace{0.2cm}q_{2}}\binom{k_{1}+q_{1}}{m-k_{2}-q_{2}}_{\ast}\frac{B(\pi+z_{1}+q_{1},\nu+n+u-1-z_{1}-q_{1})}{B(\pi+k_{1}+q_{1},\nu+n-k_{1}-q_{1})},
\end{eqnarray*}
where
$$C=C(\alpha,u,z_{1},z_{2},k_{1},k_{2})=\alpha^{z_{2}-k_{2}}(1-\alpha)^{(u-1)-(z_{1}-k_{1})-(z_{2}-k_{2})},$$
clearly does not depend on $q_{1}$ and $q_{2}$.

\noindent We will, in fact, show an even stronger fact. More precisely, we will show that, for all $n\geq2$, all $u=2,\ldots,n$, all $m=0,\ldots,n$, all $(z_{1},z_{2})\in S(z_{1},z_{2}):=\{(z_{1},z_{2})\in\{0,\ldots,n-1\}^{2}\colon z_{1}+z_{2}\leq n-1\}$, all $k_{1}\in A(z_{1},u)$ and all $k_{2}\in A_{z_{1},k_{1}}(z_{2},u)$, one has
$$\sigma(n,u,m,z_{1},k_{1},k_{2})=0,$$
where

\begin{eqnarray*}
&& \sigma(n,u,m,z_{1},k_{1},k_{2})\\
&& :=\sum_{q_{1}=0}^{u}\sum_{q_{2}=0}^{u-q_{1}}(-1)^{q_{1}}\binom{u}{q_{1}\hspace{0.2cm}q_{2}}\binom{k_{1}+q_{1}}{m-k_{2}-q_{2}}_{\ast}\frac{B(\pi+z_{1}+q_{1},\nu+n+u-1-z_{1}-q_{1})}{B(\pi+k_{1}+q_{1},\nu+n-k_{1}-q_{1})}.
\end{eqnarray*}

\noindent Towards this aim, we first show that

\begin{eqnarray}
 \label{semplificata}
&& \sigma(n,u,m,z_{1},k_{1},k_{2}) =\\
&&=\sum_{q_{1}=0}^{u}(-1)^{q_{1}}\binom{u}{q_{1}}\frac{B(\pi+z_{1}+q_{1},\nu+n+u-1-z_{1}-q_{1})}{B(\pi+k_{1}+q_{1},\nu+n-k_{1}-q_{1})}\sum_{q_{2}=0}^{u-q_{1}}\binom{u-q_{1}}{q_{2}}\binom{k_{1}+q_{1}}{m-k_{2}-q_{2}}_{\ast}\notag\\
&&=\binom{k_{1}+u}{m-k_{2}}_{\ast}\sum_{q_{1}=0}^{u}(-1)^{q_{1}}\binom{u}{q_{1}}\frac{B(\pi+z_{1}+q_{1},\nu+n+u-1-z_{1}-q_{1})}{B(\pi+k_{1}+q_{1},\nu+n-k_{1}-q_{1})}.\notag
\end{eqnarray}
In other words, we show that
\begin{equation}
\label{sorpresa}
\tilde{\sigma}_{q_{1}}(m,u,k_{1},k_{2})=\sum_{q_{2}=0}^{u-q_{1}}\binom{u-q_{1}}{q_{2}}\binom{k_{1}+q_{1}}{m-k_{2}-q_{2}}_{\ast}=\binom{k_{1}+u}{m-k_{2}}_{\ast},
\end{equation}
i.e. that $\tilde{\sigma}_{q_{1}}(m,u,k_{1},k_{2})$ does, actually, not depend on $q_{1}$. For reading convenience, set $u-q_{1}=i$ and $m-k_{2}=j$ and rewrite $\tilde{\sigma}_{q_{1}}(m,u,k_{1},k_{2})$ as
$$\tilde{\sigma}_{i}(j,u,k_{1})=\sum_{q_{2}=0}^{i}\binom{i}{q_{2}}\binom{k_{1}+u-i}{j-q_{2}}_{\ast}.$$
To see that
$$\tilde{\sigma}_{i}(j,u,k_{1})=\binom{k_{1}+u}{j}_{\ast},$$
start by fixing $i\in\{0,\ldots,u\}$. If $j<0$ (i.e. if $m<k_{2}$), then the equality is trivial. If $0\leq j\leq i$, then, by definition of $\binom{k_{1}+q_{1}}{j-q_{2}}_{\ast}$, we have
$$\tilde{\sigma}_{i}(j,u,k_{1})=\tilde{\sigma}_{j}(j,u,k_{1}),$$
and the result follows by a direct application of the van der Monde formula. If $j\geq i$, then a direct application of the classical Pascal's triangle gives

\begin{eqnarray*}
\binom{k_{1}+u}{j}_{\ast}&=&\binom{k_{1}+u-1}{j}_{\ast}+\binom{k+u-1}{j-1}_{\ast}\\
&=&\binom{k_{1}+u-2}{j}_{\ast}+2\binom{k+u-2}{j-1}_{\ast}+\binom{k_{1}+u-2}{j-2}_{\ast}\\
&=&\binom{k_{1}+u-3}{j}_{\ast}+3\binom{k+u-3}{j-1}_{\ast}+3\binom{k_{1}+u-3}{j-2}_{\ast}+\binom{k_{1}+u-3}{j-3}_{\ast}\\
&=&\cdots\\
&=&\tilde{\sigma}_{i}(j,u,k_{1}).
\end{eqnarray*}

\noindent Now that (\ref{sorpresa}) is in order, fix, arbitrarily, $n,u,z_{1},k_{1}$. For all $k_{2}$ and $m$ such that $m\in\{0,\ldots,k_{2}-1\}\cup\{k_{1}+u+1,\ldots,n\}$, we have, by definition of $\binom{k_{1}+u}{m-k_{2}}_{\ast}$, that $\sigma(n,u,m,z_{1},k_{1},k_{2})=0$ as desired. The fact that this is still the case for all $m\in\{k_{2},\ldots,k_{1}+u\}$, follows from (\ref{semplificata}) and Lemma \ref{sommedentro}. The proof of Theorem \ref{t:main} is concluded.

\smallskip

\noindent \textbf{Acknowledgements.}
The authors are especially grateful to Eugenio Regazzini for a number of illuminating discussions. I. Pr\"unster is also affiliated to Collegio Carlo Alberto and partially supported by MIUR, Grant 2008MK3AFZ, and Regione Piemonte.

\end{document}